\documentclass[11pt, a4paper]{amsart}

\usepackage{amssymb}
\usepackage{amsmath}
\usepackage{enumerate}
\swapnumbers

\theoremstyle{plain}

\newtheorem{theorem}{Theorem}[section]
\newtheorem{corollary}[theorem]{Corollary}
\newtheorem{lemma}[theorem]{Lemma}
\newtheorem{proposition}[theorem]{Proposition}

\theoremstyle{definition}

\newtheorem{definition}[theorem]{Definition}
\newtheorem{example}[theorem]{Example}
\newtheorem{examples}[theorem]{Examples}
\newtheorem{exercise}[theorem]{Exercise}

\newtheorem{conjecture}[theorem]{Conjecture}

\theoremstyle{remark}

\newtheorem{remark}[theorem]{Remark}

\def\R{{\mathbb R}}
\def\C{{\mathbb C}}

\def\N{{\mathbb N}}
\def\Z{{\mathbb Z}}

\def\F{{\mathbb F}}

\DeclareMathOperator{\graph}{graph}

\DeclareMathOperator{\ran}{ran}
\DeclareMathOperator{\dom}{dom}
\DeclareMathOperator{\rep}{Rep}
\DeclareMathOperator{\diag}{diag}
\DeclareMathOperator{\I}{I}
\DeclareMathOperator{\II}{II}
\DeclareMathOperator{\III}{III}
\DeclareMathOperator{\vN}{vN}

\newcommand{\actson}{\curvearrowright}

\title[Set theory and von Neumann algebras]{Set theory and von Neumann algebras}

\author{Asger T\"ornquist}

\author{Martino Lupini}

\begin{document}

\maketitle

\section*{Introduction}

The aim of the lectures is to give a brief introduction to the area of von Neumann algebras to a typical set theorist. The \emph{ideal} intended reader is a person in the field of (descriptive) set theory, who works with group actions and equivalence relations, and who is familiar with the rudiments of ergodic theory, and perhaps also orbit equivalence. This should not intimidate readers with a different background: Most notions we use in these notes will be defined. The reader \emph{is} assumed to know a small amount of functional analysis. For those who feel a need to brush up on this, we recommend consulting \cite{gkp}.

What is the motivation for giving these lectures, you ask. The answer is two-fold: On the one hand, there is a strong connection between (non-singular) group actions, countable Borel equivalence relations and von Neumann algebras, as we will see in Lecture 3 below. In the past decade, the knowledge about this connection has exploded, in large part due to the work of Sorin Popa and his many collaborators. Von Neumann algebraic techniques have lead to many discoveries that are also of significance for the actions and equivalence relations themselves, for instance, of new cocycle superrigidity theorems. On the other hand, the increased understanding of the connection between objects of ergodic theory, and their related von Neumann algebras, has also made it possible to construct large families of non-isomorphic von Neumann algebras, which in turn has made it possible to prove \emph{non-classification} type results for the isomorphism relation for various types of von Neumann algebras (and in particular, factors).

For these reasons, it seems profitable that (descriptive) set theorists should know more about von Neumann algebras, and we hope that these lectures can serve as a starting point for those who want to start learning about this wonderful area of mathematics.

Before moving on to the mathematics, a warning or three: Theorems and Lemmas below that are not attributed are \emph{not} due to the authors, and their origin can usually be deduced by perusing the surrounding text. Mistakes, however, are entirely due to the authors (specifically, the first author). On the other hand, it is often implicitly assumed below that Hilbert spaces are separable, and if a result is false {\it without} this assumption, then it is only a mistake if it also is false {\it with} this assumption.

\section*{Lecture 1}

\addtocounter{section}{1}

\begin{center}

{\it Basic definitions, examples, and the double commutant theorem.}

\end{center}

\medskip

In the following, $H$ will denote a separable Hilbert space with inner
product $\left\langle \cdot ,\cdot \right\rangle $ and norm $\left\Vert
\cdot \right\Vert $, while $B(H) $ will denote the linear space
of bounded linear operators on $H$. Unless otherwise specified, $H$ is also
assumed to be infinite-dimensional. Define on $B\left( H\right) $ the 
\emph{operator norm}
\begin{equation*}
\left\Vert T\right\Vert =\sup \left\{ \left\Vert T\xi \right\Vert \left\vert
\,\xi \in H,\left\Vert \xi \right\Vert \leq 1\right. \right\} \text{.}
\end{equation*}
Endowed with this norm, $B(H)$ is a Banach space, which is
nonseparable, unless $H$ is finite dimensional. Define the \emph{weak
(operator) topology }on $B(H) $ as the topology induced by the
family of complex valued functions on $B(H)$
\begin{equation*}
T\mapsto \left\langle T\xi ,\eta \right\rangle \text{,}
\end{equation*}
where $\xi$ and $\eta$ range over $H$. The \emph{strong (operator) topology }on $B\left(
H\right) $ is instead defined as the topology induced by the family of
functions
\begin{equation*}
T\mapsto \left\Vert T\xi \right\Vert \text{,}
\end{equation*}
where $\xi$ ranges over $H$. Both of these topologies are Polish on $B^{1}(H)$, where $B^{1}(H) $ denotes the unit ball of $B\left(
H\right) $ with respect to the operator norm. The Borel structure induced by either of these topologies coincide (see exercise \ref{e.borelstruct} below), and $B(H)$ is a standard Borel space with this Borel structure. (Note though that $B(H)$ is Polish in neither of these topologies.) The weak topology is weaker than the strong topology, which in turn is weaker than the norm topology. If $T\in B(H)$, denote by $
T^{\ast}$ the {\it adjoint} of $T$, i.e. unique element of $B(H)$ such that for all $\xi ,\eta \in H$,
\begin{equation*}
\left\langle T\xi ,\eta \right\rangle =\left\langle \xi ,T^{\ast }\eta
\right\rangle.
\end{equation*}
Endowed with this operation, $B(H)$ turns out to be a Banach
$*$-algebra, which furthermore satisfies the ``C*-axiom'',
\begin{equation*}
\left\Vert T^{\ast }T\right\Vert =\left\Vert T\right\Vert ^{2}\text{,}
\end{equation*}
that is, $B(H)$ is a C*-algebra (in the abstract sense). A subset of $B(H)$ will be called \emph{self-adjoint} if it contains the adjoint of any of its elements. A self-adjoint subalgebra of $B(H)$ will also be called a \emph{$*$-subalgebra.}

\begin{definition}\label{d.cstar-wstar}\ 
\begin{enumerate}
\item A \emph{C*-algebra }is a norm closed $*$-subalgebra of $B(H)$ (for some $H$).

\item A \emph{von Neumann algebra} is a weakly closed self-adjoint subalgebra of 
$B(H)$ (for some $H$) containing the identity operator $I$.
\end{enumerate}
\end{definition}

Since the weak topology is weaker than the norm topology, a von Neumann
algebra is, in particular, a C*-algebra. However, the ``interesting'' von Neumann algebras all turn out to be non-separable with respect to the norm topology. This should be taken as an indication that it would not be fruitful to regard von Neumann algebras simply as a particular kind of C*-algebras; rather, von Neumann algebras have their own, to some degree separate, theory.\footnote{This is not to suggest that no knowledge about C*-algebras will be useful when studying von Neumann algebras.}

The usual catch-phrase that people attach to definition \ref{d.cstar-wstar} to underscore the difference between C*-algebras and von Neumann algebras is that \emph{C*-algebra theory is non-commutative topology} (of locally compact spaces, presumably), while \emph{von Neumann algebra theory is non-commutative measure theory.} This grows out of the observation that all commutative C*-algebras are isomorphic to the $*$-algebra $C_0(X)$ of all complex-valued continuous functions on some locally compact $X$ (with the sup-norm), whereas all commutative von Neumann algebras are isomorphic to $L^\infty(Y,\nu)$ for some $\sigma$-finite measure space $(Y,\nu)$ (see example \ref{e.vnalgs}.(1) below). The reader is encouraged to ponder what the locally compact space $X$ would look if we wanted to realize the isomorphism $C(X)\simeq L^\infty(Y,\nu)$, where $Y=\N$ with the counting measure, or even $Y=[0,1]$ with Lebesgue measure. (This should help convince the skeptic why von Neumann algebras merit having their own special theory.)

\begin{remark}
Von Neumann algebras were introduced by Murray and von Neumann in a series of papers \cite{mvn1,mvn2,vn3,mvn4} in the 1930s and 1940s, where the basic theory was developed. In older references, von Neumann algebras are often called \emph{W*-algebras.}
\end{remark}

\begin{exercise}
Show that the map $T\mapsto T^*$ is weakly continuous, but not strongly continuous. ({\it Hint}: Consider the unilateral shift.) Show that the map $(S,T)\mapsto ST$ is separately continuous (with respect to the weak or strong topology), but not jointly continuous. Conclude that if $A\subseteq B(H)$ is a $*$-subalgebra of $B(H)$, then the weak closure of $A$ is a $*$-subalgebra of $B(H)$.
\end{exercise}

The reader may also want to verify that the map $(S,T)\mapsto ST$ is continuous on $B^1(H)$ w.r.t. the strong topology, but that this fails for the weak topology.
\begin{examples}\label{e.vnalgs}\ 

\begin{enumerate}

\item Let $\left( M,\mu \right) $ be a $\sigma $-finite standard measure space. Each $
f\in L^{\infty }\left( X,\mu \right) $ gives rise to a bounded operator $
m_{f}$ on $L^{2}\left( X,\mu \right) $, defined by
\begin{equation*}
\left( m_{f}\left( \psi \right) \right) \left( x\right) =f\left( x\right)
\psi \left( x\right) \text{.}
\end{equation*}
The set $\left\{ m_{f}\left\vert \,f\in L^{\infty }\left( X,\mu
\right) \right. \right\} $ is an abelian von Neumann algebra, which may be seen to
be a maximal abelian subalgebra of $B\left( L^{2}\left( M,\mu \right)
\right) $ (see Exercise \ref{e.masa}). It can be shown that any abelian von Neumann algebra looks like this, see \cite[III.1.5.18]{blackadar}.

\item $B(H)$ is a von Neumann algebra. In particular, when $H$ has
finite dimension $n$, $B(H)$ is the algebra $M
_{n}(\C)$ of $n\times n$ matrices over the complex
numbers.

\item Suppose that $\Gamma $ is a countable discrete group and consider
the unitary operators on $L^{2}(\Gamma)$ defined by
\begin{equation*}
\left( U_{\gamma }\psi \right) \left( \delta \right) =\psi \left( \gamma
^{-1}\delta \right) \text{.}
\end{equation*}
Define $\mathcal{A}$ to be the self-adjoint subalgebra of $B\left( L^{2}\left(
\Gamma \right) \right) $ generated by $\left\{ U_{\gamma }\left\vert
\,\gamma \in \Gamma \right. \right\} $. Observe that an element of $\mathcal{
A}$ can be written as $\sum_{\gamma \in \Gamma }a_{\gamma }U_{\gamma }$,
where $\left( a_{\gamma }\right) _{\gamma \in \Gamma }$ is a $\Gamma $
-sequence of elements of $\mathbb{C}$ such that $\left\{ \gamma \in \Gamma
\left\vert \,a_{\gamma }\neq 0\right. \right\} $ is finite. The weak closure
of $\mathcal{A}$ is a von Neumann algebra, called the \textbf{group von
Neumann algebra }$L\left( \Gamma \right) $ of $\Gamma $. For each $\delta \in \Gamma $ define
\begin{equation*}
e_{\delta }\left( \gamma \right) =
\begin{cases}
1 & \text{if $\gamma =\delta $,} \\ 
0 & \text{otherwise.}
\end{cases}
\end{equation*}
For $x=\sum_{\gamma \in \Gamma }a_{\gamma }U_{\gamma }\in \mathcal{A}$, define $\tau:\mathcal A\to\C$ by
\begin{equation*}
\tau(\sum_{\gamma \in \Gamma }a_{\gamma }U_{\gamma})=\left\langle \left( \sum_{\gamma \in \Gamma }a_{\gamma }U_{\gamma }\right)
e_{1},e_{1}\right\rangle =a_{1}\text{,}
\end{equation*}
where $1$ is the identity element of $\Gamma$. Clearly, $\tau$ is a linear functional on $\mathcal A$, which is positive ($\tau(x^*x)\geq 0$) and $\tau(I)=1$ (that is, $\tau$ is a {\bf state} on $\mathcal A$). A direct calculation shows that if $y=\sum_{\delta \in \Gamma }b_{\delta }U_{\delta }\in \mathcal{A}$ is another finite sum, then
\begin{equation*}
\left\langle \left( \sum_{\gamma \in \Gamma }a_{\gamma }U_{\gamma }\right)\left( \sum_{\delta \in \Gamma }b_{\delta }U_{\delta }\right)
e_{1},e_{1}\right\rangle = \sum_{\gamma,\delta\in\Gamma, \gamma\delta=1} a_\gamma b_\delta.
\end{equation*}
Thus $\tau(xy)=\tau(yx)$, i.e., $\tau$ is a \emph{trace} on $\mathcal A$. Of course, the formula $\tau(x)=\langle xe_1,e_1\rangle$ makes sense for any operator in $B(L^2(\Gamma))$ and defines a state on $B(L^2(\Gamma))$, and the reader may verify (using that composition is separately continuous) that $\tau$ satisfies the trace property $\tau(xy)=\tau(yx)$ for any $x,y$ in the weak closure of $\mathcal A$. Note that $\tau$ is weakly continuous.
\end{enumerate}

\end{examples}

\begin{definition}
For $X\subseteq B(H)$, define the \textbf{commutant} of $
X$ to be the set
$$
X^\prime=\{T\in B(H): (\forall S\in X) TS=ST\}.
$$
\end{definition}

Note that (by separate continuity of composition), $X^\prime$ is both weakly and strongly closed. Furthermore, if $X$ is self-adjoint, then so is $X^\prime$, and thus $X^\prime$ is a von Neumann algebra when $X$ is self-adjoint. In particular, the \textbf{double commutant} $X^{\prime\prime}=(X^\prime)^\prime$ is a von Neumann algebra containing $X$.

\medskip

The next theorem, known as \emph{double commutant theorem}, is a
cornerstone of the basic theory of von Neumann algebras. It is due to von
Neumann.

\begin{theorem}\label{t.dct}
Let $A\subseteq B(H)$ be a self-adjoint subalgebra of $B(H)$ which contains the identity operator. Then $A$ is strongly dense in $A^{\prime\prime}$.
\end{theorem}

\begin{corollary}
If $A$ is a self-adjoint subalgebra of $B(H)$ containing $I$, then the weak and strong closure of $A$ coincide with $A^{\prime\prime}$, i.e. $\overline{
A}^{so}=\overline{A}^{wo}=A^{\prime \prime}$.
\end{corollary}

The corollary follows from the double commutant theorem by noting that $\overline{
A}^{so}\subseteq\overline{A}^{wo}\subseteq A^{\prime \prime}$. To prove the double commutant theorem, we start with the following:

\begin{lemma}\label{l.dct}
Let $A$ be a selfadjoint subalgebra of $B(H)$ which contains $I$, and let $T_0\in A^{\prime\prime}$. Then for any $\xi \in H$ and $\varepsilon>0$ there is $S\in A$ such that $\|(T_0-S) \xi \| <\varepsilon$.
\end{lemma}

\begin{proof}
Define $p\in B(H)$ to be the orthogonal projection onto the
closure of the subspace $A\xi=\{ T\xi : T\in A\}$ (which contains $\xi$).
We claim that $p\in A^{\prime }$. To see this, note that if $\eta \in \ker(
p)$ and $S\in A$, then for every $T\in A$ we have
\begin{equation*}
\langle S\eta , T\xi \rangle = \langle \eta ,S^{\ast }T\xi\rangle =0,
\end{equation*}
whence $\ker(p)$ is $A$-invariant. It follows that $p\in A^\prime$, since for any $T\in A$ and $\eta\in H$ we have
$$
pT\eta=pTp\eta+pT(1-p)\eta=pTp\eta=Tp\eta.
$$
Since $T_0\in A^{\prime\prime}$ we therefore have $T_0\xi=T_0p\xi=pT_0\xi\in\ran(p)$, and so there is some $S\in A$ such that $\|S\xi-T_0\xi\|<\varepsilon$, as required.
\end{proof}

\begin{proof}[Proof of Theorem \ref{t.dct}]
We first need some notation. Fix $n\in\N$. If $(T_{i,j})_{1\leq i,j\leq n}$, where $T_{i,j}\in B(H)$, is a ``matrix of operators'', then an operator in $B(H^n)$ is defined by ``matrix multiplication'', i.e.,
$$
(S_{(T_{i,j})}(\eta_1,\ldots,\eta_n))_k=\sum_{j=1}^n T_{k,j}\eta_j.
$$
(Note that every operator in $B(H^n)$ has this form.) For $T\in B(H)$, we let $\diag(T)$ be the operator
$$
\diag(T)(\xi_1,\ldots,\xi_n)=(T\xi_1,\ldots,T\xi_n),
$$
corresponding to the matrix of operators which has $T$ on the diagonal, and is zero elsewhere.

Now fix $\xi_1,\ldots, \xi_n\in H$, $\varepsilon>0$ and $T_0\in A^{\prime\prime}$, and let $\vec{\xi}=(\xi_1,\ldots,\xi_n)$. Note that $\diag(A)=\{\diag(T):T\in A\}$ is a self-adjoint subalgebra of $B(H^n)$ containing $I$. It suffices to show that $\diag(T_0)\in \diag(A)^{\prime\prime}$ since then by Lemma \ref{l.dct} there is $S\in A$ such that $\|(\diag(S)-\diag(T_{0}))\vec{\xi}\|_{H^n}<\varepsilon$. Thus $\|S\xi_k-T_0\xi_k\|<\varepsilon$ for all $k\leq n$, which shows that $A$ is strongly dense in $A^{\prime\prime}$.

To see that $\diag(T_0)\in \diag(A)^{\prime\prime}$, simply note that $S\in \diag(A)^\prime$ precisely when it has the form $S_{(T_{i,j})}$ where $T_{i,j}\in A^\prime$ for all $i,j$. Thus $\diag(T_0)\in\diag(A)^{\prime\prime}$.
\end{proof}

\begin{remark}
If $X\subseteq B(H)$ is self-adjoint and contains $I$, then it follows from the double commutant theorem that $X^{\prime\prime}$ is the smallest von Neumann algebra containing $X$, i.e., $X^{\prime\prime}$ is the von Neumann algebra generated by $X$.
\end{remark}

An element $u$ of $B\left( H\right) $ is called {\bf partial isometry} if $u^{\ast
}u$ and (hence) $uu^{\ast }$ are orthogonal projections. The first one is
called the {\bf support projection} of $u$ and the latter the {\bf range projection} of $u$.

Recall the {\bf polar decomposition} theorem for bounded operators: If $T\in B(H)$, then there is a partial isometry $u$ such that $T=u|T|$, where $
|T|=(T^{\ast }T)^{\frac{1}{2}}$. In this
case, the support projection of $u$ is the orthogonal projection onto $\ker(T)^{\bot }$, while the range projection of $u$ is the orthogonal projection onto $\overline{\ran(T)}$. It is easy to see that $u$ restricted to $\ran(u^{\ast }u)$
is an isometry onto $\ran(uu^{\ast })$. The polar decomposition of an operator is unique.

\begin{exercise}
Let $M\subseteq B(H)$ be a von Neumann algebra. Let $T\in M$, and let $T=u|T|$ be the polar decomposition of $T$. Prove that $|T|\in M$ and $u\in M$. ({\it Hint}: To show $|T|\in M$, you may want to recall how the existence of the square root is proved, see e.g. \cite{gkp}. To show that $u\in M$, use the polar decomposition theorem.)
\end{exercise}

\begin{exercise}
Show that every operator in $B(H^n)$ has the form $S_{(T_{i,j})}$ for some matrix of operators $(T_{i,j})$ in $B(H)$. ({\it Hint}: For a given operator $S\in B(H^n)$, consider the operators $p_jSp_i$ where $p_i$ is the projection onto the $i$'th coordinate.) 
\end{exercise}

\begin{exercise}\label{e.borelstruct}
Show that the Borel structure on $B^1(H)$ induced by the weak and strong topologies coincide. Conclude that the Borel structure induced by the weak and strong topologies on $B(H)$ coincide, and is standard.
\end{exercise}

\begin{exercise}\label{e.masa}
(See example \ref{e.vnalgs}.1.) Let $(X,\mu)$ be a $\sigma$-finite measure space. Show that $L^\infty(X,\mu)$ is a maximal Abelian $*$-subalgebra of $B(L^2(X,\mu))$ (when identified with the set of multiplication operators $\{m_f:f\in L^\infty(X,\mu)\}$). Conclude that $L^\infty(X,\mu)$ is strongly (and weakly) closed, and so is a von Neumann algebra. ({\it Hint}: Take $T\in L^\infty(X,\mu)^\prime$, and argue that the function $T(1)$ is essentially bounded.)
\end{exercise}

\section*{Lecture 2}

\addtocounter{section}{1}
\setcounter{theorem}{0}

\begin{center}

{\it Comparison theory of projections, type classification, direct integral decomposition, and connections to the theory of unitary group representations.}

\end{center}

\medskip

An element $p\in B(H)$ is called a \textbf{projection} (or more precisely, an \emph{orthogonal} projection) if
\begin{equation*}
p^{2}=p^{\ast }=p.
\end{equation*}
The reader may easily verify that this is equivalent to (the more geometric definition) $p^2=p$ and $\ker(p)^\perp=\ran(p)$. ({\it Warning}: From now on, when we write ``projection'' we will always mean an orthogonal projection. This is also the convention in most of the literature.) For projections $p,q\in B(H)$, write  $p\leq q$ if $\ran(p)\subseteq\ran(q)$. We will say that $p$ is a {\bf subprojection} of $q$. (One may more generally define $T\leq S$ iff $S-T$ is a positive operator for any $T,S\in B(H)$. This definition agrees with our definition of $\leq$ on projections.)

\medskip

Unlike their C*-algebra brethren\footnote{There are examples of C*-algebras with no non-trivial projections!}, von Neumann algebras always have \emph{many} projections. The key fact is this: A von Neumann algebra $M$ contains all spectral projections (in the sense of the spectral theorem) of any normal operator $T\in M$. It follows from this that a von Neumann algebra is generated by its projections (see exercise \ref{e.specprojn} below). It is therefore natural to try to build a structure theory of von Neumann algebras around an analysis of projections.

\begin{definition}
Let $M$ be a von Neumann algebra. We let $P(M)$ be the set of projections in $M$. For $p,q\in P(M)$, we say that $p$ and $q$ are {\bf Murray-von Neumann
equivalent} (or simply {\bf equivalent}) if there is a partial isometry $u\in M$
such that $u^{\ast }u=p$ and $uu^{\ast }=q$. We write $p\sim q$. We will say that $p$ is {\bf subordinate} to $q$, written $p\precsim q$, if $p$ is equivalent to a subprojections of $q$
(i.e., there is $p'\leq q$ such that $p\sim p'$).

\end{definition}

\begin{example}
In $B(H)$, $p\precsim q$ iff the range of $p$ has dimension
smaller than or equal to the range of $q$. Thus, the ordering on $P\left(
B\left( H\right) \right)\left/ \sim \right. $ is linear, and isomorphic to $
\left\{ 0,1,...,n\right\} $ if $\dim H=n$ and to $\mathbb{N}\cup \left\{
\infty \right\} $ if $H$ is infinite dimensional.
\end{example}

The previous example highlights why it must be emphasized that the definition of $\sim$ is ``local'' to the von Neumann algebra $M$ (i.e., that we require $u\in M$ in the definition). Otherwise, $\sim$ would only measure the dimension of the range of the projection $p$, which does not depend on $M$ in any way. However, the idea that $\sim$ and $\precsim$ are somehow related to dimension \emph{relative to $M$} is essentially correct (though the precise details are subtle). For instance, in $L(\F_n)$ (where $\F_n$ is the free group on $n>1$ generators) it turns out that $P(L(\F_n))/\sim$ ordered linearly by $\precsim$, and is order-isomorphic to $[0,1]$. So in some sense, we need a continuous range to measure  dimension in $L(\F_n)$. More about this later.

\medskip

It is clear that $\precsim $ is a transitive relation with $I$ being a maximal element, and $0$ being the minimal. Furthermore, we have:

\begin{proposition}[``Schr\"oder-Bernstein for projections'']\label{p.sb}
If $p\precsim q$ and $q\precsim p$, then $p\sim q$.
\end{proposition}
\begin{exercise}
Prove Proposition \ref{p.sb}.
\end{exercise}

\begin{definition}
Let $M$ be a von Neumann algebra. A projection $p\in M$ is said to be

\begin{itemize}
\item \textbf{finite} if it is not equivalent to a proper subprojection
of itself;

\item \textbf{infinite} if it is not finite;

\item \textbf{purely infinite} if it has no nonzero finite subprojections;

\item \textbf{semifinite} if it is not finite, but is the supremum
of an increasing family of finite projections.

\item \textbf{minimal} if it is non-zero, and has no proper non-zero subprojections.
\end{itemize}

A von Neumann algebra $M$ will be called finite (infinite, purely infinite, or semifinite) if the identity $I\in M$ is finite (respectively, infinite, purely infinite, or semifinite).
\end{definition}

\begin{proposition}\label{p.factor}
Let $M$ be a von Neumann algebra. Then the following are equivalent:
\begin{enumerate}
\item The \emph{center} $\mathcal{Z}(M) =M\cap M^{\prime}$ of $M$ consists of scalar multiples of the identity $I$, (i.e., $\mathcal{Z}(M)=\C I$).
\item $P(M)/\sim$ is linearly ordered by $\precsim$.
\end{enumerate}
\end{proposition}

The proof is outlined in exercise \ref{e.factor}.

\begin{definition}
A von Neumann algebra $M$ is called a \emph{factor} if (1) (and therefore (2)) of the previous proposition holds.
\end{definition}

It is not hard to see that $M_n(\C)$ and $B(H)$ are factors. Another source of examples are the group von Neumann algebras $L(\Gamma)$, when $\Gamma$ is an {\it infinite conjugacy class} (or {\it i.c.c.}) group, meaning that the conjugacy class of each $\gamma\in\Gamma\setminus\{1\}$ is infinite.

The next theorem shows that factors constitute the building blocks of von Neumann algebras, as all (separably acting, say) von Neumann algebras can be decomposed into a generalized direct sum (i.e., integral) of factors. This naturally shifts the focus of the theory to factors, rather than general von Neumann algebras. Quite often, a general theorem that can be proven for factors can then be extended to all von Neumann algebras using the direct integral decomposition.

\begin{theorem}
Let $M$ be a von Neumann algebra. Then there is a standard $\sigma $-finite
measure space $\left( X,\mu \right)$, a Borel field $\left( H_{x}\right)
_{x\in X}$ of Hilbert spaces and a Borel field $(M_x)_{x\in X}
$ of von Neumann algebras in $B(H_x)$ such that: $H$ is isomorphic to $\int_X H_x d\mu(x)$, and identifying $H$ and $\int_X H_x d\mu(x)$, we have
\begin{enumerate}
\item $M_x$ is a factor for all $x\in X$;
\item $M=\int_{X} M_{x}d\mu(x)$;
\item $\mathcal Z(M)=L^\infty(X,\mu)$.
\end{enumerate}
Moreover, this decomposition is essentially unique.
\end{theorem}

The proof is rather involved; it is given in full detail in \cite{oanielsen} (see also \cite{blackadar,dixmier}). Some comments about the theorem are in place, however. A Borel field of Hilbert spaces is nothing but a standard Borel space $X$ with a partition $X=\bigsqcup_{n=0,1,2,\ldots,\N} X_n$, and the vectors in $H$ are Borel functions $f:X\to\ell^2(\N)$, where for $x\in X_n$ we have that $f(x)$ is in the space generated by the first $n$ standard basis vectors of $\ell^2(\N)$, and the inner product is given by $\langle f,g\rangle =\int \langle f(x),g(x)\rangle d\mu(x)$. Equality $M=\int M_x d\mu(x)$ means that every operator $T\in M$ can be written as $T=\int T_xd\mu(x)$ (i.e., $(Tf)=T_x(f(x))$ $\mu$-a.e.), where $T_x\in M_x$, and $x\mapsto T_x$ is Borel (in the obvious sense) w.r.t. the $\sigma$-algebra generated by the weakly open sets.

\begin{definition}[Type classification of factors]
Let $M$ be a factor. We say that $M$ is
\begin{enumerate}[\indent (a)]

\item type $\I_n$ if it is finite and $P(M)/\!\sim$ is order isomorphic to $n=\{0,1,\ldots, n-1\}$ with the usual order;

\item type $\I_\infty$ if it is infinite and $P(M)
/\!\sim$ is order isomorphic to $\mathbb{N}\cup \{\infty
\}$ (where $n<\infty$ for all $n\in\N$);

\item type $\II_{1}$ if it is finite and $P(M)/\!\sim$ is order isomorphic to $[0,1]$;

\item type $\II_{\infty }$ if it is semifinite and $P(M)/\!\sim$ is order isomorphic to $[ 0,\infty]$;

\item type $\III$ if it is purely infinite and $P(M)/\!\sim$ is order isomorphic to $\left\{ 0,\infty \right\} $.
\end{enumerate}
\end{definition}

The reader may object that $[0,1]$ and $[0,\infty]$ are order-isomorphic, and so are $\{0,1\}$ and $\{0,\infty\}$. The intention is that the $\infty$ indicates that the $\precsim$-maximal projection is infinite, and so the notation contains additional information. It is clear that any factor $M$ can be at most one of the types, but even more so, the list is in fact exhaustive and complete:

\begin{theorem}
Every factor is either type $\I_n$, $\I_\infty$, $\II_1$ and $\II_\infty$ or type $\III$. Moreover, there is at least one factor of each type.
\end{theorem}

That the types list all possibilities is not too hard to see. That there is an example of each type is much harder. (We will see examples of all types in the last lectures.)

Up to isomorphism, the algebra $M_n(\C)$ of $
n\times n$ complex matrices is the only type $I_{n}$ von Neumann algebra,
while $B(H)$ for $H$ infinite dimensional (and separable!) is the only type $
I_{\infty}$ factor. Such an easy description of the isomorphism classes of type $\II$ and type $\III$ factors is not possible: As we will see below, the type $\II$ and $\III$ factors cannot even be classified up to isomorphism by countable structures!

The cornerstone of the theory of $\II_1$ factors (or, more generally, finite von Neumann algebras) is the existence of a (unique) faithful trace, which is continuous on the unit ball. We have already seen an example of a trace when we discussed $L(\Gamma)$.  The trace on a $\II_1$ factor is an invaluable technical tool when working with projections in a $\II_1$ factor.

\begin{theorem}[Murray-von Neumann]

(A) Every $\II_1$ factor $M$ has a faithful normal {\bf trace}, i.e., a positive linear functional $\tau:M\to\C$ which is
\begin{enumerate}
\item {\bf normal}, meaning that $\tau$ is weakly continuous on the unit ball of $M$;
\item {\bf faithful}, meaning that for $x\in M$, $x=0$ iff $\tau(x^*x)=0$;

\item {\bf tracial},  meaning that $\tau(xy)=\tau(yx)$ for all $x,y\in M$.
\end{enumerate}
The trace is unique up to a scalar multiple.

(B) Every $\II_\infty$ factor admits a faithful normal semifinite trace defined on the set $M^+$ of positive operators in $M$. That is, there is an additive map $\tau:M^+\to [0,\infty]$ which satisfy $\tau(rx)=r\tau(x)$ for all $r>0$, which is continuous (in the natural sense) on the unit ball of $M^+$ with respect to the weak topology, and which is faithful and tracial. The semifinite trace on a $\II_\infty$ factor is unique up to multiplication by a scalar.
\end{theorem}

The trace gives us a measure of the size of projections, much like the dimension function $p\mapsto\dim(\ran(p))$ does on the projections in $B(H)$. However, for a $\II_1$ factor $M$ we have that $\tau(P(M))=[0,\tau(I)]$, and so the projections on a $\II_1$ factor have ``continuous dimension'' (understood locally in the $\II_1$ factor, of course). Because of the trace property, two projections $p,q\in P(M)$ are equivalent iff $\tau(p)=\tau(q)$. The existence of a trace characterizes the finite factors: $M$ is finite iff there is a trace as above on $M$.

\begin{exercise}\label{e.specprojn}
Let $M$ be a von Neumann algebra, and let $T\in M$ be a normal operator. Show that all the \emph{spectral projections}, i.e., projections in the Abelian von Neumann algebra generated by $T$, belong to $M$. Conclude that $M$ is generated by its projections, i.e., $M=P(M)^{\prime\prime}$. ({\it Hint}: Any operator can be written $T=\Re(T)+i\Re(-iT)$, where $\Re(T)=\frac 1 2 (T+T^*)$.)
\end{exercise}

\begin{exercise}
Prove that if $p$ and $q$ are finite projections in a factor $M$ and $p\sim q$, then $1-p\sim 1-q$. Conclude that there is a unitary operator $U\in M$ such that $U^*p U=q$. ({\it Hint}: Quickly dispense with the case that $M$ is infinite.)
\end{exercise}

\begin{exercise}
If $H$ is infinite dimensional, then $B(H)$ does not admit a trace which is weakly continuous on $B^1(H)$.
\end{exercise}

The next two exercises are harder.

\begin{exercise}
Any two elements $p,q$ of $P\left( M\right) $ have $\inf $ and $\sup $ in $
P\left( M\right) $ with respect to the relation $\leq $. Denoting these by $
p\wedge q$ and $p\vee q$ respectively, one has
\begin{equation*}
p\vee q-p\sim q-p\wedge q.
\end{equation*}
This is known as \emph{Kaplanski's identity.} ({\it Hint:} Consider the domain and range $(1-q)p$. A proof can be found in \cite[Theorem 6.1.7]{kadring2}.)
\end{exercise}

\begin{exercise}\label{e.factor}
Prove Proposition \ref{p.factor}. (A proof can be found in \cite[Theorem 6.2.6]{kadring2}.)
\end{exercise}

\subsection*{{\sc Applications to unitary representations.}} To illustrate the usefulness of the notions of factors and comparison of projections, we briefly turn to study unitary representation of countable discrete groups. What is said in this section applies almost without change to unitary representations of locally compact groups and representations of C*-algebras as well.

\begin{definition}
Let $\Gamma $ be a countable discrete group. A \textbf{unitary
representation} of $\Gamma $ on $H$ is a homomorphism $\pi :\Gamma
\to U(H)$, where $U(H)$ is the group of unitary operators in $B(H)$.
\end{definition}

For notational convenience, we will write $\pi_\gamma$ for $\pi(\gamma)$. Let $\pi$ and $ \sigma$ be unitary representations of $\Gamma$ on Hilbert spaces $H_0,H_1$ respectively. We define
$$
R_{\pi,\sigma}=\{T\in L(H_0,H_1): (\forall \gamma\in\Gamma) T\pi_\gamma=\sigma_\gamma T\},
$$
where $L(H_0,H_1)$ is the set of bounded linear maps $H_0\to H_1$. A map $T\in R_{\pi,\sigma}$ is called an \emph{intertwiner} of $\pi$ and $\sigma$. We say that $\pi$ and $\sigma$ are {\bf disjoint}, written $\pi\perp\sigma$, if $R_{\pi,\sigma}=\{0\}$. We let $R_\pi=R_{\pi,\pi}$. Note that $R_\pi=\{\pi_\gamma:\gamma\in\Gamma\}^{\prime}$, thus $R_\pi$ is a von Neumann algebra.

If $p\in P(R_\pi)$, then $\ran(p)$ is a $\pi$-invariant closed subspace in $H$; conversely, a projection onto any closed $\pi$-invariant subspace must be in $R_\pi$. Note further that if $p,q\in P(R_\pi)$ and $p\sim q$ (in $R_\pi$), then $\pi\upharpoonright \ran(p)$ and $\pi\upharpoonright \ran(q)$ are isomorphic. Thus $R_\pi$ gives us useful information about the invariant subspaces of $\pi$. For instance, the reader may easily verify that $\pi$ is irreducible (i.e., has no non-trivial invariant subspaces) iff $R_\pi=\C I$. More generally, if $p\in P(R_\pi)$ is a minimal non-zero projection (i.e., having no proper non-zero subprojections in $P(R_\pi)$), then $\pi\upharpoonright \ran(p)$ is irreducible.

It is tempting to hope that any unitary representation can be written as a direct sum, or direct integral, of irreducible unitary representations. It can, but the decomposition is very badly behaved (and non-unique) unless $\Gamma$ is abelian by finite. A better behaved decomposition theory is based around the following notion:

\begin{definition}
A unitary representation $\pi $ of $\Gamma $ is called a {\bf factor representation} (or sometimes a {\bf primary} representation) if $R_{\pi}$ is a factor. A representation $\pi$ is said to be type $\I_n$, $\I_\infty$, $\II_1$, $\II_\infty$ or $\III$ according to what type $R_\pi$ is. Similarly, it is called finite, infinite, semifinite or purely infinite according to what $R_\pi$ is.
\end{definition}

The decomposition theory for unitary representation can now be obtained from the decomposition of von Neumann algebras into factors.

\begin{theorem}[Mackey]\label{t.mackey}
Let $\pi$ be a unitary representation of $\Gamma$ on a separable Hilbert space. Then there is a standard $\sigma $-finite measure space $(X,\mu)$, a Borel field of Hilbert spaces $(H_x)_{x\in X}$, and a Borel field $(\pi _x)_{x\in X}$ of factor representations of $\Gamma$ such that
\begin{equation*}
\pi \simeq \int_{X}\pi _{x}d\mu \left( x\right) 
\end{equation*}
and if $x\neq y$, then $\pi_x\perp\pi_y$. Furthermore, this decomposition of $\pi$ is essentially unique.
\end{theorem}

\begin{exercise}
Show that a type $\I_n$ factor representation ($n\in\N\cup\{\infty\}$) is a direct sum of $n$ irreducible representations.
\end{exercise}

For type $\II$ and $\III$ factor representations, nothing like this is true, since there are no minimal non-zero projections. In the case of type $\III$ factors, the restriction of $\pi$ to any invariant non-zero closed subspace (i.e., what may be called a ``piece'' of the representation) is isomorphic to the whole representation. In the case when $\pi$ is type $\II$, every piece of $\pi$ may be subdivided into $n$ smaller pieces that are all isomorphic to each other, but unlike the type $\III$ case, the trace (and the associated notion of dimension) gives us a sense of the size of the pieces of $\pi$. (So, in the type $\III$ everyone can get as much $\pi$ as they want, in the type $\II$ case we can divide the $\pi$ into however many albeit small pieces we like, and in the type $\I$ case there is a minimal size of the pieces of $\pi$ (and maximal number of pieces, too).)

We close the section with mentioning the following recent theorem. It solves an old problem of Effros, who asked if the conjugacy relation $\simeq$ for unitary representation of a fixed countable discrete group $\Gamma$ is Borel (in the space $\rep(\Gamma,H)=\{\pi\in U(H)^\Gamma: \pi\text{ is a homomorphism}\}$). 

\begin{theorem}[Hjorth-T\"ornquist, 2011]
Let $H$ be an infinite dimensional separable Hilbert space and let $\Gamma$ be a countable discrete group. The conjugacy relation in $\rep(\Gamma,H)$ is $F_{\sigma\delta}$.
\end{theorem}

The proof uses only classical results that have been introduced already above. To give a rough sketch of what happens in the proof, fix $\pi, \sigma\in\rep(\Gamma,H)$. The idea is that it is an $F_{\sigma\delta}$ statement to say that there are pieces (i.e., projections $p\in P(R_{\pi})$) arbitrarily close to $I$ (in the weak topology, say) such that $\pi\upharpoonright \ran(p)$ is isomorphic to a sub-representation of $\sigma$. Call this statement $S_0(\pi,\sigma)$. One first proves that for factor representations $\pi$ and $\sigma$, it holds that $\pi\simeq\sigma$ iff $S_0(\pi,\sigma)$ and $S_0(\sigma,\pi)$. This is trivial in the purely infinite (type $\III$) case, while in finite and semifinite case the trace gives us a way of proving that if pieces of $\pi$ closer and closer to $I$ can fit into $\sigma$, $\pi$ itself can fit into $\sigma$ (and conversely). It then follows that $\sigma\simeq\pi$. The proof is finished by using Theorem \ref{t.mackey}.

The details can be found in \cite{hjto12}. The theorem applies more generally to representations of locally compact second countable groups and separable C*-algebras.

\section*{Lecture 3}

\addtocounter{section}{1}
\setcounter{theorem}{0}

\begin{center}

{\it The group-measure space construction, the von Neumann algebra of a non-singular countable Borel equivalence relation, orbit equivalence, von Neumann equivalence, and Cartan subalgebras.}

\end{center}

\subsection{The group-measure space construction} Let $(X,\mu)$ be a standard $\sigma$-finite measure space, and let $\Gamma$ be a countable discrete group. A Borel action $\sigma:\Gamma\actson X$ is said to be {\bf non-singular} (w.r.t. $\mu$) if for all $\gamma\in\Gamma$ we have that $\sigma_\gamma\mu\approx\mu$ (i.e., $\mu(A)=0$ iff $\mu(\sigma_\gamma^{-1}(A))=0$), and that $\sigma$ is {\bf measure preserving} (w.r.t. $\mu$) if $\sigma_\gamma\mu=\mu$ for all $\gamma\in\Gamma$ (i.e., $\mu(A)=\mu(\sigma_\gamma^{-1}(A))$ for all measurable $A\subseteq X$). The action is {\bf ergodic} if any invariant measurable set is either null or conull; it is {\bf a.e.\! free} (or \emph{essentially} free) if for any $\gamma\in\Gamma\setminus\{1\}$ we have that
$$
\mu(\{x\in X: \sigma_\gamma(x)=x\})=0.
$$
The action $\sigma$ induces an {\bf orbit equivalence relation}, denoted $E_\sigma$, on the space $X$, which is defined by
$$
xE_\sigma y\iff (\exists\gamma\in\Gamma) \sigma_\gamma(x)=y.
$$
Note that $E_\sigma$ is a Borel subset of $X\times X$.

To a non-singular action $\sigma:\Gamma\actson (X,\mu)$ as above there is an associated von Neumann algebra. The description is slightly easier in the case when $\sigma$ is measure preserving, so we will assume that this is the case. In this case the construction is also closely parallel to the construction of the group von Neumann algebra.

Let $\nu$ be the counting measure on $\Gamma$. Give $\Gamma\times X$ the product measure $\nu\times\mu$, and define for each $\gamma\in\Gamma$ a unitary operator $U_\gamma$ on $L^2(\Gamma\times X,\nu \times \mu)$ by
\begin{equation*}
\left( U_{\gamma }\psi \right) \left( \delta ,x\right) =\psi \left( \gamma
^{-1}\delta ,\sigma _{\gamma^{-1}}(x)\right).
\end{equation*}
Further, define for each $f\in L^\infty(X,\mu)$ an operator $m_f\in B(L^2(\Gamma\times X))$ by 
\begin{equation*}
\left( m_{f}\psi \right) \left( \delta ,x\right) =f\left( x\right) \psi
\left( \delta ,x\right) \text{.}
\end{equation*}
Define
\begin{equation*}
L^\infty(X,\mu)\rtimes_\sigma\Gamma:=\left(\{m_{f} : f\in L^{\infty }(X,\mu)\}
\cup \{U_{\gamma} : \gamma \in \Gamma \}\right)^{\prime\prime}.
\end{equation*}
The von Neumann algebra $L^\infty(X,\mu)\rtimes_\sigma\Gamma$ is called the {\bf group-measure space} von Neumann algebra of the action $\sigma$. It is an example of W*-crossed product. It is worth noting that if $X$ consists of a single point then we get the group von Neumann algebra of $\Gamma$.

The group $\Gamma$ acts on $L^\infty(X,\mu)$ by $(\hat\sigma_\gamma(f))(x)=f(\sigma_{\gamma^{-1}}(x))$; note that $\hat\sigma$ is a $*$-automorphism of $L^\infty(X,\mu)$ for every $\gamma \in \Gamma $. An easy calculation shows that $U_\gamma m_f=m_{\hat\sigma_\gamma(f)}U_\gamma$, and so $U_\gamma m_f U_{\gamma}^*=m_{\hat\sigma_\gamma(f)}$. It follows that
$$
\mathcal A=\{\sum_{\gamma \in F}m_{f_{\gamma }}U_{\gamma }:F\subseteq\Gamma \text{ is finite}\wedge f_\gamma\in L^\infty(X,\mu)\}
$$
is a $*$-algebra, which is the smallest $*$-algebra containing the operators $U_\gamma$ and $m_f$, and so $\mathcal A$ is dense in $L^\infty(X,\mu)\rtimes_\sigma\Gamma$ by the double commutant theorem.

Assume now that $\mu$ is a \emph{finite} measure. Then
\begin{equation*}
e_{1}\left( \gamma ,x\right) =
\begin{cases}
1 & \text{if }\gamma =1\text{,} \\ 
0 & \text{otherwise,}
\end{cases}
\end{equation*}
defines an element of $L^{2}(\Gamma \times X)$, and we can define a state (positive linear functional) on $B(L^2(\Gamma\times X))$ by $\tau(T)=\langle Te_1,e_1\rangle$. If $x=\sum_{\gamma \in F }m_{f_{\gamma }}U_{\gamma
}\in \mathcal A$, then we have
\begin{equation*}
\tau(x)=\left\langle \left( \sum_{\gamma \in F }m_{f_{\gamma }}U_{\gamma
}\right) e_{1},e_{1}\right\rangle =\int_{X}f_{1}d\mu,
\end{equation*}
and one may easily verify that for $x,y\in\mathcal A$ we have $\tau(xy)=\tau(yx)$. Using that multiplication is separately weakly continuous in $B(L^2(\Gamma,X))$ we see that the trace property extends to the weak closure of $\mathcal A$. Thus $\tau$ is a weakly continuous trace on $L^\infty(X,\mu)\rtimes_\sigma\Gamma$ with $\tau(I)$ finite, from which it follows that $L^\infty(X,\mu)\rtimes_\sigma \Gamma$ is a finite von Neumann algebra. Taking $f=\chi_A$ to be the characteristic function of some measurable $A\subseteq X$, we see that $\tau(m_f)=\mu(A)$. So if $\mu$ is non-atomic we have $\tau(P(L^\infty(X,\mu)\rtimes_\sigma\Gamma))=[0,\tau(I)]$.

One may ask if $L^\infty(X,\mu)\rtimes_\sigma\Gamma$ can be a factor. The above analysis shows that \emph{if} it is a factor \emph{and} $\mu$ is finite and non-atomic, then $L^\infty(X,\mu)\rtimes_\sigma\Gamma$ must be a type $\II_1$ factor. Though far from providing an exhaustive answer to this question, the following is a key result in this direction:

\begin{theorem}
If $\sigma$ is a.e.\ free and ergodic, then $L^\infty(X,\mu)\rtimes_\sigma\Gamma$ is a factor.
\end{theorem}

The reader may find it amusing to verify that if we take $\sigma$ to be the action of $\Z/n\Z$ on itself by translation (which preserves the counting measure, and is free and ergodic), then the corresponding group-measure space factor is just $M_n(\C)$.

A question of central importance is to understand the relationship between the action $\sigma$ and $L^\infty(X,\mu)\rtimes_\sigma\Gamma$. Which properties of the action, if any, are reflected in the group-measure space von Neumann algebra? This is far from completely understood, and a vast body of literature addressing various special cases exists. (In general, the more ``rigid'' a group is, the more information about the action and the group will be encoded into the group-measure space algebra.)

\begin{definition}
Let $\sigma:\Gamma\actson (X,\mu)$ and $\pi:\Gamma\actson (Y,\nu)$ be measure preserving actions on standard $\sigma$-finite measure spaces. We say that $\sigma$ and $\pi$ are
\begin{enumerate}
\item {\bf conjugate} if there is a non-singular Borel bijection $T:X\to Y$ such that for all $\gamma\in\Gamma$ we have $T\sigma_\gamma(x)=\pi_\gamma T(x)$ for almost all $x\in X$.
\item {\bf orbit equivalent} if there is a non-singular Borel bijection $T:X\to Y$ such that for almost all $x,y\in X$ we have
$$
xE_\sigma y\iff T(x) E_\pi T(y).
$$
\item {\bf von Neumann equivalent} (or {\bf W*-equivalent}) if $L^\infty(X,\mu)\rtimes_\sigma\Gamma$ and $L^\infty(Y,\nu)\rtimes_\pi\Gamma$ are isomorphic.
\end{enumerate}
\end{definition}
It is clear that conjugacy implies von Neumann equivalence and orbit equivalence, but little else can be said immediately. We will see below (exercise \ref{e.gmsp}) that when the actions $\sigma$ and $\pi$ are free, then orbit equivalence implies von Neumann equivalence. 

Let $M=L^\infty(X,\mu)\rtimes_\sigma\Gamma$. A special role is played by the Abelian subalgebra generated by the operators $m_f$ (which we identify with $L^\infty(X,\mu)$ in the obvious way). It can be seen that $\sigma$ is a.e.\ free iff $L^\infty(X,\mu)$ is a maximal Abelian subalgebra. The subalgebra $L^\infty(X,\mu)$ has the property that the unitary normalizer
$$
\{U\in U(M) : U L^{\infty }(X,\mu)U^{\ast }\subseteq L^{\infty
}(X,\mu)\}
$$
generates $M$; we say that $L^\infty(X,\mu)$ is a {\bf regular} subalgebra of $M$. So when $\sigma$ is a.e.\ free, then $L^\infty(X,\mu)$ is a {\bf Cartan subalgebra} of $M$, i.e., a maximal Abelian regular subalgebra (see exercise \ref{e.gmsp} below).

Finally, we define the crossed product when $\sigma$ is non-singular, but not measure preserving. In this case the outcome will be a purely infinite von Neumann algebra. It is still the case that $\hat\sigma$ acts on $L^\infty(X,\mu)$ by $*$-automorphisms, and $L^\infty(X,\mu)$ is represented as multiplication operators on $H=L^2(X,\mu)$. Consider then the Hilbert space $L^2(\Gamma,\nu, H)$ of $L^2$ functions with values in $H$. On this Hilbert space we define the operators
$$
\left((\lambda_\gamma\psi)(\delta)\right)(x)=\psi(\gamma^{-1}\delta)(x)
$$
and
$$
(\pi_f(\psi)(\gamma))(x)=(\hat\sigma_{\gamma}f)(x)(\psi(\gamma))(x).
$$
We then let $L^\infty(X,\mu)\rtimes_\sigma\Gamma$ be the von Neumann algebra generated by this family of operators. This definition also hints at how one may go about defining the crossed product even more generally: Instead of $L^\infty(X,\mu)$, consider an arbitrary von Neumann algebra $N$ acting on a Hilbert space $H$ and let $\hat\sigma:\Gamma\actson N$ be an action on $N$ by $*$-automorphisms. The formulas above still define bounded operators on $L^2(\Gamma,\nu,H)$, and they generate a von Neumann algebra that we denote by $N\rtimes_{\hat\sigma}\Gamma$.

\subsection{The von Neumann algebra of a non-singular countable Borel equivalence relation.}

In two highly influential papers \cite{fm1,fm2}, Feldman and Moore developed a string of results related to countable Borel equivalence relations and von Neumann algebras. In the first paper they study countable non-singular Borel equivalence relations and their cohomology; in the second paper they construct a von Neumann algebra $M(E)$ directly from a non-singular countable Borel equivalence relation $E$, and study its properties. 

The present section is dedicated to the construction of $M(E)$. It can best be described as the construction of a ``matrix algebra'' over the equivalence relation. The details that are not provided below can for the most part be found in \cite{fm2}.\footnote{Both \cite{fm1} and \cite{fm2} are extremely well written and are warmly recommended.}

\medskip

Before we can define $M(E)$, we need a few facts from the first paper \cite{fm1}. The first of these is by now widely known:

\begin{theorem}[\cite{fm1}]\label{t.fm1}
If $E$ is a countable Borel equivalence relation on a standard Borel space $X$, then there is a countable group $\Gamma$ of Borel automorphisms of $X$ which induce $E$. (That is, $E=E_\sigma$ for some Borel action $\sigma:\Gamma\actson X$.)
\end{theorem}

Now fix a countable Borel equivalence relation $E$ on a $\sigma$-finite standard measure space $(X,\mu)$. On $E$, which is a Borel subset of $X^2$, we define two Borel measures
\begin{equation*}
\mu _{\ast }(A) =\int |A_{x}| d\mu(
x) 
\end{equation*}
and
\begin{equation*}
\mu^{\ast}(A) =\int|A^{y}| d\mu(y)
\end{equation*}
for each Borel $A\subseteq E$. Here, as usual, $A_x=\{y\in X: (x,y)\in A\}$ and $A^y=\{x: (x,y)\in A\}$. One now has the following:

\begin{theorem}[\cite{fm1}]\label{t.nonsingular}
Let $E$ and $(X,\mu)$ be as above. Then the following are equivalent:
\begin{enumerate}
\item $\mu_*$ and $\mu^*$ are absolutely equivalent, i.e., $\mu_*\approx\mu^*$.
\item There is a countable group of non-singular Borel automorphisms of $X$ which induce $E$.
\item Any Borel automorphism whose graph is contained in $E$ preserves the measure class of $\mu$.
\end{enumerate}
\end{theorem}

We will say that $E$ is {\bf non-singular} (w.r.t. $\mu$) if one (and all) of the conditions in Theorem \ref{t.nonsingular} holds. When this is the case, let $D(x,y)=\frac{d\mu_*}{d\mu^*}(x,y)$ be the Radon-Nikodym derivative. From now on, we will always assume that $E$ is non-singular.

\begin{exercise}
Prove that if $f\in L^{\infty }(E,\mu_*)$, then
\begin{equation*}
\int fd\mu _{\ast }=\int \left( \sum_{y\in \left[ x\right] _{E}}f\left(
x,y\right) \right) d\mu(x),
\end{equation*}
and similarly that for $f\in L^\infty(E,\mu^*)$ we have
\begin{equation*}
\int fd\mu^*=\int \left( \sum_{x\in \left[ y\right] _{E}}f\left(
x,y\right) \right) d\mu(y).
\end{equation*}
\end{exercise}

\begin{definition}
A function $a\in L^\infty(E,\mu^*)$ is called {\bf left finite} if there is $n\in\N$ such that for $\mu^*$-almost all $(x,y)\in E$ we have
$$
|\{z: a(x,z)\neq 0\}|+|\{z: a(z,y)\neq 0\}|\leq n.
$$
\end{definition}
If $a$ and $b$ are left finite we define the product $ab$ in analogy to matrix multiplication,
$$
(ab)(x,y)=\sum_{z\in [x]_E} a(x,z)b(z,y),
$$
and we also define the ``adjoint matrix'' $a^*$ in the natural way, $a^*(x,y)=\overline{a(y,x)}$ (complex conjugation). The following is easily verified:
\begin{lemma}
The left finite functions are stable under sum, scalar multiplication, product and adjoint, and so they form a $*$-algebra.
\end{lemma}
For each left finite $a\in L^\infty(E,\mu^*)$ we can define an operator $L_a$ on $L^2(E,\mu^*)$ by
$$
\left( L_{a}\psi \right) \left( x,y\right) =\sum_{z\in [x]_E} a(x,z)\psi(z,y). 
$$
\begin{lemma}
Every operator of the form $L_a$ is bounded when $a$ is left finite. Moreover, for all $a,b\in L^\infty(E,\mu^*)$ left finite we have $L_a L_b=L_{ab}$ and $L_a^*=L_{a^*}$. Thus $\mathcal A=\{L_a:a\text{ is left finite}\}$ forms a $*$-subalgebra of $B(L^2(E,\mu^*))$.
\end{lemma}
\begin{proof}
Fix a left finite function $a$ and $n$ such that $|\{z: a(x,z)\neq 0\}|+|\{z: a(z,y)\neq 0\}|\leq n$ for almost all $(x,y)\in E$. We will show that $\|L_a\|\leq n\|a\|_\infty$. For $\psi\in L^2(E,\mu^*)$ we have
\begin{equation}\label{eq.fmbd}
\|(L_a\psi)\|^2=\int\left(\sum_{x\in [y]_E} |(L_a\psi)(x,y)|^2\right)d\mu(y)=\int\left(\sum_{x\in [y]_E} |\sum_{z\in [y]_E} a(x,z)\psi(z,y)|^2\right)d\mu(y).
\end{equation}
Fix a typical $y$. For $x\in [y]_E$ fixed, we get from the Cauchy-Schwarz inequality and left finiteness of $a$ that
$$
|\sum_{z\in [y]_E} a(x,z)\psi(z,y)|^2\leq n\sum_{z\in [y]_E} |a(x,z)\psi(z,y)|^2\leq n\|a\|_\infty^2\sum_{z\in [y]_E, a(x,z)\neq 0} |\psi(z,y)|^2.
$$
Further,
$$
\sum_{x\in [y]_E}\sum_{z\in [y]_E, a(x,z)\neq 0} |\psi(z,y)|^2=\sum_{z\in [y]_E}\sum_{x\in [y]_E, a(x,z)\neq 0} |\psi(z,y)|^2\leq n \sum_{z\in [y]_E}|\psi(z,y)|^2,
$$
where the last inequality follows since $a$ is left finite. Combining this with \eqref{eq.fmbd} we get
$$
\|(L_a\psi)\|^2\leq n^2\|a\|_\infty^2\int\left(\sum_{z\in [y]_E}|\psi(z,y)|^2\right) d\mu(y)=n^2\|a\|_\infty^2\|\psi\|^2,
$$
as required. The remaining claims are left for the reader to verify.
\end{proof}

\begin{definition}
We define
$$
M(E)=\{L_a:a\text{ is left finite}\}^{\prime\prime}
$$
and call this the von Neumann algebra of the equivalence relation $E$.
\end{definition}

\begin{exercise}
Show that if on $n=\{0,1\ldots,n-1\}$ we take $E=n\times n$, then $M(E)\simeq M_n(\C)$. Also, describe $M(E)$ with other choices of $E\subseteq n\times n$.
\end{exercise}

It is easy to see that $M(E)$ only depends on $E$ up to orbit equivalence, i.e., if $F$ is a non-singular countable Borel equivalence relation on $(Y,\nu)$ and there is a non-singular Borel bijection $T$ such that $xE y\iff T(x)FT(y)$ a.e., then $M(F)\simeq M(E)$ as von Neumann algebras.

Following what is standard notation, we denote by $[E]$ the group of all non-singular bijections $\phi:X\to X$ which satisfy $xE\phi(x)$ a.e., and we let $[[E]]$ denote the semigroup of all \emph{partial} $\phi:A\to B$, where $A,B\subseteq X$ are Borel sets, which satisfy $xE\phi(x)$ for a.a.\ $x\in\dom(\phi)$. Given $\phi\in [[E]]$ and a function $f\in L^\infty(X,\mu)$, a left finite function $a_{\phi,f}$ is defined by
$$
a_{\phi,f}\left( x,y\right) =
\begin{cases}
f(x) & \text{if }y=\phi(x),\\ 
0 & \text{otherwise.}
\end{cases}
$$
When $f=1$, the constant $1$ function, then we will write $a_{\phi}$ for $a_{\phi,1}$. A direct calculation shows that 
\begin{equation}\label{eq.covrel}
a_{\phi_0,f}a_{\phi_1,g}=a_{\phi_1\circ\phi_0,(g\circ\phi_0) f}.
\end{equation} 
In particular, letting $\Delta(x)=x$ for all $x\in X$, we have that $a_{\Delta,f}a_{\Delta,g}=a_{\Delta,fg}$. It follows that $f\mapsto L_{a_{\Delta,f}}$ provides an embedding of $L^\infty(X,\mu)$ into $M(E)$. That is, $L^\infty(X,\mu)$ is naturally identified with the ``diagonal matrices'' in $M(E)$. When there is no danger of confusion we will therefore write $f$ for $L_{a_{\Delta,f}}$ and write $L^\infty(X,\mu)$ for $\{L_{a_{\Delta,f}}:f\in L^\infty(X,\mu)\}$.

We assume from now on that $E$ is \emph{aperiodic}, i.e., that all $E$ classes are infinite. It is then clear from Theorem \ref{t.fm1} that one can find a sequence $(\phi_n)_{n\in\N_0}$ in $[E]$ whose graphs are pairwise disjoint and $E=\bigsqcup_{n\in\N_0} \graph(\phi_n)$. It is practical to always assume that $\phi_0=\Delta$. The following lemma provides a useful standard form for the operators in $M(E)$.

\begin{lemma}
Let $(\phi_{n})_{n\in\N_0}$ be a sequence in $\left[ E\right]$  with pairwise disjoint graphs whose union is $E$. Any element $x\in M(E)$ can be written uniquely as
$$
x=\sum_{n=0}^{\infty }f_n L_{a_{\phi_n}}
$$
for some sequence $(f_{n})_{n\in\N_0}$ in $L^{\infty}(X,\mu)$ and with convergence in the weak topology. In particular,
\begin{equation}\label{eq.generate}
M(E)=\left(L^\infty(X,\mu)\cup\{L_{a_{\phi_n}}:n\in\N_0\}\right)^{\prime\prime}.
\end{equation}

\end{lemma}
The proof, which we will skip, can be done by hand and is not that hard. The reader should be warned, though, that the Lemma does not provide any information about \emph{which} sequences $(f_n)$ define operators in $M(E)$ in this way. It does allow us to prove two key results about $M(E)$.

\begin{theorem}
$L^{\infty }(X,\mu)$ is a Cartan subalgebra of $M(E)$.
\end{theorem}

\begin{proof}
We first prove that $L^\infty(X,\mu)$ is a maximal Abelian subalgebra. For this, suppose $T\in L^\infty(X,\mu)^\prime \cap M(E)$. Using the previous lemma, write $T=\sum_{n=0}^\infty f_n L_{a_{\phi_n}}$. Let $g\in L^\infty(X,\mu)$. Using that $gT=Tg$ and \eqref{eq.covrel} above we get
$$
\sum_{n=0}^\infty g f_n L_{a_{\phi_n}}=\sum_{n=0}^\infty f_n L_{a_{\phi_n}}g=\sum_{n=0}^\infty f_n (g\circ\phi_n) L_{a_{\phi_n}},
$$
and so from the uniqueness of the expansion we have $f_n(g-g\circ\phi_n)=0$. Thus if $f_n(x)\neq 0$ it follows that $g(x)=g(\phi_n(x))$ for \emph{any} $g\in L^\infty(X,\mu)$, which means that $\phi_n(x)=x$ whenever $f_n(x)\neq 0$. But this shows that $T=f_0L_{a_{\phi_0}}=f_0L_{a_{\Delta}}$, that is, $T\in L^\infty(X,\mu)$.

To see that $L^{\infty}(X,\mu)$ is regular in $M(E)$, observe that the normalizer of $L^{\infty}(X,\mu)$ in $M(E)$ contains the unitary elements of $L^{\infty}(X,\mu)$ as well as all $L_{a_{\phi}}$ for $\phi\in \left[ E\right] $, and so generates $M(E)$.
\end{proof}

\begin{theorem}
$M(E)$ is a factor iff $E$ is ergodic.
\end{theorem}

\begin{proof}
Suppose first that $M(E)$ is a factor, and let $f\in L^\infty(X,\mu)$ be an $E$-invariant function. It follows from \eqref{eq.generate} that $f\in\mathcal Z(M(E))$, thus $f$ is a constant multiple of $1$. Conversely, if $E$ is ergodic, let $x\in\mathcal Z(M(E))$. From the previous theorem we know that $\mathcal Z(M(E))\subseteq L^\infty(X,\mu)$, and so $x=f$ for some $f\in L^\infty(X,\mu)$. Since $f$ is central we have that $fL_{a_\phi}=L_{a_\phi} f$ for all $\phi\in [E]$, and so $f\circ\phi=f$ for all $\phi\in E$. Whence $f$ is $E$ invariant, and therefore constant, which shows that $\mathcal Z(M(E))=\C 1$.
\end{proof}

With a little more effort one can go on to prove the following interesting theorem:

\begin{theorem}
Let $E$ and $F$ be non-singular countable Borel equivalence relations on standard $\sigma$-finite measure spaces $(X,\mu)$ and $(Y,\nu)$, respectively. Then $E$ is orbit equivalent to $F$ if and only if the inclusions $L^\infty(X,\mu)\subseteq M(E)$ and $L^\infty(Y,\nu)\subseteq M(F)$ are isomorphic, i.e., there is an isomorphism of $M(E)$ and $M(F)$ which maps $L^\infty(X,\mu)$ onto $L^\infty(Y,\nu)$.
\end{theorem}

This illustrates that if we want to understand the relationship between $E$ and $M(E)$ we will want to understand the nature of $L^\infty(X,\mu)$ as a subalgebra of $M(E)$.

\begin{exercise}
Assume that $\mu(X)<\infty$ and that $\mu$ is $E$-invariant (which means that $\mu$ is invariant under all elements of $[E]$). Show that $\mathbf 1_{\graph(\Delta)}\in L^2(E,\mu^*)$ and that $\tau(x)=\langle x\mathbf 1_{\graph(\Delta)},\mathbf 1_{\graph(\Delta)}\rangle$ defines a trace on $M(E)$. Conclude that when $\mu$ is non-atomic and $E$ is $\mu$-ergodic, $M(E)$ is a type $\II_1$ factor.
\end{exercise}

\begin{exercise}\label{e.gmsp}
Let $\Gamma$ be a countable discrete group and let $\sigma:\Gamma\actson (X,\mu)$ be a measure preserving a.e.\ free (!) action. Show that $L^\infty(X,\mu)\rtimes_\sigma\Gamma$ is isomorphic to $M(E_\sigma)$, and in fact that the inclusions $L^\infty(X,\mu)\subseteq L^\infty(X,\mu)\rtimes_\sigma\Gamma$ and $L^\infty(X,\mu)\subseteq M(E_\sigma)$ are isomorphic. Conclude that $L^\infty(X)$ is a Cartan subalgebra of $L^\infty(X,\mu)\rtimes_\sigma\Gamma$, and that $L^\infty(X,\mu)\rtimes_\sigma\Gamma$ is a factor when $\sigma$ is a.e.\ free and ergodic. ({\it Hint}: Consider the map $\Gamma\times X\to E$ defined by $(\gamma,x)\mapsto (x,\sigma_\gamma(x))$.)
\end{exercise}

\begin{exercise}
Call a measurable function $b:E\to\C$ {\bf right finite} if $D(x,y)^{-\frac 1 2} b(x,y)$ is left finite (where $D$ is the Radon-Nikodym derivative of $\mu_*$ w.r.t. $\mu^*$). Show that
$$
R_b(\psi)(x,y)=\sum_{z\in [x]_E} \psi(x,z)b(z,y)
$$
defines a bounded operator on $L^2(E,\mu^*)$, and that $\{R_b:b\text{ is right finite}\}$ is a $*$-algebra contained in $M(E)^\prime$.
\end{exercise}

The next exercise is somewhat harder.

\begin{exercise}
With notation as in the previous exercise, show that $M(E)'=\{R_b:b\text{ is right finite}\}^{\prime\prime}$.
\end{exercise}

\section*{Lecture 4}

\addtocounter{section}{1}
\setcounter{theorem}{0}

\begin{center}

{\it Hyperfinite von Neumann algebras, ITPFI factors, Connes' classification of hyperfinite factors and Krieger's theorems, non-classification results via descriptive set theory, and rigidity.}

\end{center}

\medskip

This lecture, the last, is dedicated to giving an overview of developments in the field of von Neumann algebras in the last 40 years. In other words, if we had a whole semester's worth of lectures, these are some of the topics that would be covered in detail.

\subsection{Hyperfinite von Neumann algebras.}

\begin{definition}
A von Neumann algebra $M$ is called
\begin{enumerate}
\item {\bf finite dimensional} if it is isomorphic to $M_{n_1}(\C)\oplus\cdots\oplus M_{n_k}(\C)$.
\item \textbf{hyperfinite} if there is an increasing sequence $(M_i)$ of finite dimensional sub-algebras of $M$ such that $\bigcup M_i$ is dense in $M$.
\end{enumerate}
\end{definition}
Hyperfiniteness is equivalent to a number of other conditions on a von Neumann algebra, among them {\it amenability} and {\it injectivity} (neither of which we define here).

The class of hyperfinite von Neumann algebras and factors is very rich, and their theory has been developed further than any other general class. A useful source of examples comes from infinite {\it tensor products} of finite von Neumann algebras.

\subsection{Finite tensor products.} Let $H_1$ and $H_2$ be (complex) Hilbert spaces. We will denote by $H_1\odot H_2$ the {\it algebraic} tensor product of $H_1$ and $H_2$. Recall that this means that we have a map $H_1\times H_2\to H_1\odot H_2: (\xi,\eta)\mapsto\xi\otimes\eta$ with the property that any bilinear map $\rho:H_1\times H_2\to E$, where $E$ is some vector space over $\C$, has the form $\rho(\xi,\eta)=\hat\rho(\xi\otimes\eta)$, for some linear map $\hat\rho:H_1\odot H_2\to E$. The elements $\xi\otimes\eta$ are called (elementary) tensors, and the tensors generate $H_1\odot H_2$.

There is a unique inner product $\langle\cdot,\cdot\rangle$ on $H_1\odot H_2$ that satisfies
$$
\langle\xi_1\otimes \eta_1,\xi_2\otimes\eta_2\rangle=\langle\xi_1,\xi_2\rangle_{H_1}\langle\eta_1,\eta_2\rangle_{H_2}.
$$
The tensor product of the Hilbert spaces $H_1$ and $H_2$ is the completion of $H_1\odot H_2$ w.r.t. the norm induced by $\langle\cdot,\cdot\rangle$. It is denoted $H_1\otimes H_2$.

Let $M_i\subseteq B(H_i)$, $i=1,2$, be von Neumann algebras. For $x\in M_1$ and $y\in M_2$ we define an operator $x\otimes y$ on $H_1\otimes H_2$ by $x\otimes y(\xi\otimes\eta)=(x\xi)\otimes(y\eta)$. It is not hard to see that each such operator is bounded. We let
$$
M_1\otimes M_2=\{x\otimes y:x\in M_1, y\in M_2\}^{\prime\prime}
$$
and call this the tensor product of $M_1$ and $M_2$. The tensor product of an arbitrary but finite number of Hilbert spaces and von Neumann algebras are defined similarly.

\begin{exercise}\label{e.tensor1}
Show that $H\simeq H\otimes \C$ for any Hilbert space $\C$. More generally, if $H_1$ and $H_2$ are Hilbert spaces and $\xi\in H_2$ is a unit vector, then $\eta\mapsto \eta\otimes\xi$ is an embedding of $H_1$ into $H_1\otimes H_2$.
\end{exercise}

\begin{exercise}
Show that $M_n(\C)\otimes M_m(\C)\simeq M_{nm}(\C)$.
\end{exercise}

\subsection{Infinite tensor products} Let $(H_i)_{i\in\N}$ be a sequence of Hilbert spaces and for each $i\in\N$ let $\xi_i\in H_i$ be a unit vector. (We will call a pair $(H,\xi)$ where $\xi$ is a unit vector in the Hilbert space $H$ a \emph{pointed} Hilbert space. The vector $\xi\in H$ will be called a \emph{base point}.) For each $n\in\N$ we have by exercise \ref{e.tensor1} that $\bigotimes_{i=1}^n H_i$ can be identified naturally with the subspace
$$
\{\eta_1\otimes\cdots\otimes \eta_n\otimes\xi_{n+1}:(\forall i\leq n) \eta_i\in H_i\}
$$
of $\bigotimes_{i=1}^{n+1} H_i$. Let $H_\infty$ be the inductive limit of the system
$$
H_1\overset{\eta\mapsto \eta\otimes \xi_2}{\hookrightarrow} H_1\otimes H_2\overset{\eta\mapsto \eta\otimes\xi_3}\hookrightarrow H_1\otimes H_2\otimes H_3\hookrightarrow\cdots
$$
which is an inner product space in the natural way; we let 
$$
\bigotimes_{i=1}^n H_i\to \bigotimes_{i=1}^\infty(H_i,\xi_i): \eta\mapsto\eta\otimes\xi_{n+1}\otimes\xi_{n+2}\otimes\cdots
$$
denote the canonical embedding. It is clear that $\bigotimes_{i=1}^\infty(H_i,\xi_i)$ is generated by the ``elementary tensors'' $\eta_1\otimes\eta_2\otimes\cdots$, where $\eta_i\in H_i$, and where $\eta_i=\xi_i$ eventually. We define the {\bf infinite tensor product of the pointed Hilbert spaces} $(H_i,\xi_i)_{i\in\N}$ to be the completion of $H_\infty$. It is denoted $\bigotimes_{i=1}^\infty (H_i,\xi_i)$.

Now let $(M_i)_{i\in\N} $ be a von Neumann algebra acting on $H_i$, for each $i\in\N$. If $x_i\in M_i$, $i\in\N$, is a sequence such that $x_i=I$ eventually, then a bounded operator $\bigotimes_{i\in\N} x_i$ on $\bigotimes_{i=1}^\infty(H_i,\xi_i)$ is defined by requiring that
$$
(\bigotimes_{i\in\N} x_i)(\eta_1\otimes\eta_2\otimes\cdots)=x_1\eta_1\otimes x_2\eta_2\otimes\cdots.
$$
We define {\bf infinite tensor product of the von Neumann algebras $(M_i)_{i\in\N}$} w.r.t. the vectors $\xi_i$ as
$$
\bigotimes_{i=1}^\infty (M_i,\xi_i)=\{\bigotimes_{i\in\N} x_i: x_i\in M_i\text{ and } x_i=I \text{ eventually}\}^{\prime\prime}.
$$
{\it Warning}: As we will see below, the isomorphism type of the infinite tensor product is \emph{highly} sensitive to the asymptotic behavior of the sequence $(\xi_i)_i$.

\begin{examples}\ 
\begin{enumerate}
\item Let $(X,\mu)$ and $(Y,\nu)$ be standard probability spaces. It is well-known that $L^2(X,\mu)\otimes L^2(Y,\nu)\simeq L^2(X\times Y,\mu\times\nu)$, and that the isomorphism is given by $f\otimes g\mapsto fg$. It follows that
$$
L^\infty(X,\mu)\otimes L^\infty(Y,\nu)\simeq L^\infty(X\times Y,\mu\times\nu),
$$
where we represent each $L^\infty$ as multiplication operators on the corresponding $L^2$ space. This is easily generalized: If $(X_i,\mu_i)$ are standard probability spaces and we let $\xi_i=1$ be the constant 1 function in $L^2(X_i,\mu_i)$, then the tensor product $\bigotimes (L^2(X_i,\mu_i),\xi_i)$ is isomorphic to $L^2(\prod X_i,\prod\mu_i)$, and $\bigotimes (L^\infty(X_i,\mu_i),\xi_i)$ is isomorphic to $L^\infty(\prod X_i,\prod\mu_i)$.

Now let $\xi\in L^2(Y,\nu)$ be a unit vector (not necessarily 1.) An isometric embedding of $L^2(X,\mu)$ into $L^2(X\times Y,\mu\times\nu)$ is given by $f\mapsto f\xi$, and this corresponds to the embedding $f\mapsto f\otimes\xi$ of $L^2(X,\mu)$ into $L^2(X,\mu)\otimes L^2(Y,\nu)$. Consider also the space $L^2(X\times Y,\mu\times|\xi|^2\nu)$, and the corresponding embedding $f\mapsto f1$, where $1$ is the constant 1 function in $L^2(Y,|\xi^2|\nu)$. It is easily verified that the map
$$
L^2(X\times Y,\mu\times\nu)\to L^2(X\times Y,\mu\times|\xi|^2\nu): h\mapsto \frac h \xi
$$
is an isometry which conjugates the two embeddings of $L^2(X,\mu)$ described above.

Now let $(X_i,\mu_i)_{i\in\N}$ be a sequence of standard probability spaces, and let $\xi_i\in L^2(X_i,\mu_i)$ be a unit vector for each $i\in\N$. Our observations above now show that the tensor product $\bigotimes (L^2(X_i,\mu_i),\xi_i)$ is isomorphic to $L^2(\prod X_i,\prod|\xi_i|^2\mu_i)$ and that $\bigotimes (L^\infty(X_i,\mu_i),\xi_i)$ is isomorphic to $L^\infty(\prod X_i,\prod|\xi_i^2|\mu_i)$. It is of course well-known that the product of infinitely many probability measures is highly sensitive to the asymptotic behavior of the sequence of measures, and so this example indicates that the infinite tensor product of von Neumann algebras also is highly sensitive in this way, for much the same reasons.

\medskip

\item Let $n_i\in\N$ be an infinite sequence of natural numbers. The easiest way to represent the matrix algebra $M_{n_i}(\C)$ is to have it act on itself by matrix multiplication on the left. (If nothing else, this manoeuvre allows us to avoid introducing separate notation for the Hilbert spaces we act on.) Let $\tau:M_{n_i}(\C)\to\C$ be the normalized trace, and let $\langle x,y\rangle_\tau=\tau(y^*x)$ be the associated inner product, which makes $M_{n_i}(\C)$ into a Hilbert space. Let $\xi_i\in M_{n_i}(\C)$ be a unit vector. There is no loss in assuming that $\xi_i$ is a diagonal matrix, $\xi_i=\diag(\lambda_{i,1},\ldots,\lambda_{i,n_i})$. We can now form the infinite tensor product of the matrix algebras $M_{n_i}(\C)$ w.r.t. the vectors $\xi_i$ as described above. It turns out that it is a factor. (It can be shown that in general the tensor product (finite or infinite) of factors is again a factor.) A factor of this form is called an {\bf ITPFI factor}, which stands for \emph{infinite tensor product of factors of type $\I$}. It is clear that ITPFI factors are hyperfinite. The sequence $(\lambda^2_{i,j})$ are called the \emph{eigenvalue sequence} for the ITPFI factor. Those ITPFI factors where for some $n\in\N$ we have $n_i=n$ for all $i\in\N$ are called ITPFI$_n$ factors.

In much the same way as above, we can carry out an analysis of the nature of the inclusions of $M_n(\C)$ into $M_n(\C)\otimes M_m(\C)$ and how it depends on a choice of $\xi\in M_m(\C)$. As we will see below, the isomorphism type of an ITPFI factor depends on the eigenvalue list in a somewhat similar way to how the infinite product of finite measure spaces depends on the measure of the atoms.

\item Certain ITPFI$_2$ factors deserve special mention for their importance in the field. Pride of place goes to the {\bf hyperfinite $\II_1$ factor}. This factor arises from taking $\xi_i=(\frac 1 {\sqrt 2},\frac 1 {\sqrt 2})$ for all $i\in\N$. The hyperfinite $\II_1$ factor is perhaps the most studied and most important of all $\II_1$ factors. It is usually denoted $\mathcal R$ (or $\mathcal R_1$). It is an almost canonical presence, showing up virtually everywhere in the field.

Another important class of ITPFI$_2$ factors are the {\bf Powers factors}. Let $0<\lambda<1$, and let $\xi_i=\left((\frac \lambda {1+\lambda})^{\frac 1 2},(1-\frac \lambda {1+\lambda})^{\frac 1 2}\right)$ for all $i\in\N$. The resulting ITPFI$_2$ factor is denoted $\mathcal R_\lambda$. In 1967, it was shown by Powers in \cite{pow67} that the family $(\mathcal R_\lambda)_{\lambda\in (0,1)}$ constitutes a family of mutually non-isomorphic factors. Previous to that result, it had not been known if there were uncountably many non-isomorphic factors!
\end{enumerate}

A far-reaching study of ITPFI factors was undertaken by Araki and Woods in the paper \cite{arwo}, where they introduced a number of invariants and produced new classes of uncountably many non-isomorphic factors.

\subsection{Classification} 

The towering achievement of von Neumann algebra theory from the 1970s is Connes' complete classification of hyperfinite (or, more correctly, injective) factors, \cite{connes76}. The classification divides type $\III$ into subcases type $\III_\lambda$, $\lambda\in [0,1]$. Connes showed that there is a unique injective (hyperfinite) factor in each of the classes $\II_1$, $\II_\infty$ and $\III_\lambda$, $0<\lambda<1$. In the case $\III_0$, there are many non-isomorphic factors, but they are classified completely by an associated invariant called the \emph{flow of weights} (see e.g. \cite{connes94}). The remaining case, the type $\III_1$ case, proved to be a difficult problem in its own right. It was eventually solved by Haagerup some years later in \cite{haag}: Haagerup showed that there is a unique injective factor of type $\III_1$.

Another phenomenal achievement of the period are the results of Krieger in \cite{krieger76}. Krieger showed that every hyperfinite factor is of the form $M(E)$ for some \emph{amenable} non-singular ergodic countable Borel equivalence relation $E$. (Here we use the language of the Feldman-Moore construction (from Lecture 3) to state Krieger's results, though the work of Feldman and Moore postdates Krieger's work.) Krieger moreover showed that for \emph{amenable} ergodic non-singular countable Borel equivalence relations $E$ and $F$, $M(E)$ is isomorphic to $M(F)$ if and only if $E$ and $F$ are orbit equivalent.
\end{examples}

\begin{exercise}
Let $\Z\actson 2^\N$ be the odometer action (i.e., adding one with carry modulo 2), and let $E_0$ denote the induced equivalence relation. Let $\alpha_i\in (0,1)$ for all $i\in\N$, and let $\lambda_{i,0}=\alpha_j$ and $\lambda_{i,1}=1-\alpha_j$. Give $2^\N$ the product measure $\prod_{i=1}^\N (\lambda_{i,0}\delta_0+\lambda_{i,1}\delta_1)$, where $\delta_i$ is the Dirac measure concentrating on $i\in\{0,1\}$. Show that $E_0$ is measure class preserving, and that $M(E_0)$ is isomorphic to the ITPFI$_2$ with eigenvalue list $(\lambda_{i,j})$. In particular, if we take $\alpha_i=\frac 1 2$ for all $i\in\N$, so that $E_0$ is measure preserving, we get the hyperfinite $\II_1$ factor $\mathcal R$. (This exercise may be quite hard.)
\end{exercise}

\subsection{Non-classification} The invariant provided by Connes' classification is a certain non-singular $\R$-action (flow), which is considered up to conjugacy. This is hardly a very simple invariant. It is therefore natural to ask: How difficult is it to classify factors up to isomorphism? It is of course well-known among set theorists that such questions can be fruitfully attacked through the concept of Borel reducibility.\footnote{We will assume that the reader is familiar with the notion of Borel reducibility, and related concepts like smooth, non-smooth, classification by countable structures, etc. We refer to \cite{sato09b} for a brief overview of these notions.} We will end these lectures by giving an overview of what is known about the classification problem for factors from this point of view. The reader can find a more detailed survey of these results in \cite{sato09b}.

Before going on, we remark that the problem of classifying von Neumann algebras and factors acting on a separable Hilbert space fits nicely into the framework of descriptive set theory. Namely, a von Neumann algebra $N\subseteq B(H)$ can be identified with $N\cap B^1(H)$. The unit ball of $B(H)$ is a compact Polish space in the weak topology, and so the space $K(B^1(H))$ of compact subsets of $B^1(H)$ is itself a Polish space. One can now show that the set $\vN(H)=\{N\cap B^1(H)\in K(B^1(H)): N\text{ is a von Neumann algebra}\}$ is a Borel set. We think of $\vN(H)$ as the {\bf standard Borel space of separably acting von Neumann algebras}. The subset $\mathcal F(H)\subseteq\vN(H)$ of factors turns out to be Borel. (There are other ways of arriving at the space $\vN(H)$, and even a natural choice of Polish topology on it; see \cite{hawi98,hawi00} for an exhaustive study. The space $\vN(H)$ was introduced and studied by Effros in the papers \cite{effros65, effros66}, who also proved, among many other things, that $\mathcal F(H)$ is Borel.)

\medskip

The first known non-classification result in the area is due to Woods:

\begin{theorem}[Woods, \cite{woods}]
The classification of hyperfinite factors is not smooth. More precisely, there is a Borel reduction of $E_0$ to the isomorphism relation of ITPFI$_2$ factors (of type $\III_0$).
\end{theorem}

This seems to have been the only result of its kind that was known until a few years ago, when the following was proven:

\begin{theorem}[Sasyk-T\"ornquist, \cite{sato09a}]\label{t.sato1}
Let $\Lambda$ be the class of separably acting factors of type $\II_1$, $\II_\infty$, or $\III_\lambda$, $0\leq \lambda\leq 1$. Then the isomorphism relation in $\Lambda$ does not admit classification by countable structures.
\end{theorem}

Prior to this, it was apparently not even known if the isomorphism relation was non-smooth in any of these classes, except the type $\III_0$ case where it follows from Woods' theorem.\footnote{For a long time it was not even known if there were infinitely many non-isomorphic $\II_1$ factors. This problem was solved by McDuff in \cite{mcduff}.} In the same paper, the following was also shown:

\begin{theorem}[Sasyk-T\"ornquist, \cite{sato09a}]\label{t.sato2}
Isomorphism of countable graphs is Borel reducible to isomorphism of $\II_1$ factors.
\end{theorem}

The same construction gives the result for type $\II_\infty$ as well, but the proof falls short of handling the type $\III$ cases. It follows from the above that isomorphism of factors (of type $\II_1$ and $\II_\infty$) is analytic, and not Borel in the space $\vN(H)$.

The proofs of theorems \ref{t.sato1} and \ref{t.sato2} only became possible due to enormous advances in the understanding of the relation between measure preserving countable Borel equivalence relations and their corresponding von Neumann algebra. These advances were spearheaded by Sorin Popa, as well as a number of his collaborators, who in the last two decades have developed a large number of techniques and results that make it possible to analyze group-measure space factors in the presence of certain ``rigidity'' properties. We refer to \cite{popa,vaes} for an overview of these developments.

None of the factors that were constructed to prove theorems \ref{t.sato1} and \ref{t.sato2} are hyperfinite. So what about the classification of hyperfinite type $\III_0$ factors? Can they be classified by countable structures? The answer is again no. This was already shown in \cite{sato09a}, by showing that a standard construction of an injective type $\III_0$ factor with a prescribed flow of weights is a Borel construction. Subsequently, this result was improved to show the following non-classification result for ITPFI$_2$ factors:

\begin{theorem}[Sasyk-T\"ornquist, \cite{sato10}]
The isomorphism relation for ITPFI$_2$ factors is not classifiable by countable structures.
\end{theorem}

The proof of this is probably the most elementary and direct of all non-classification results for factors. It only relies on techniques that essentially go back to Araki and Woods in \cite{arwo}, as well as straight-forward Baire category/turbulence arguments.

\medskip

We close with a brief discussion of an open problem. The only upper bound known about the classification of von Neumann algebras (and factors) is the following:

\begin{theorem}[Sasyk-T\"ornquist, \cite{sato09a}]\label{t.sato3}
The isomorphism relation in $\vN(H)$ is Borel reducible to an orbit equivalence relation induced by a continuous action of the unitary group $U(\ell^2(\N))$ on a Polish space.
\end{theorem}

The proof is not hard, and it is natural to ask if this upper bound on the complexity of the isomorphism relation for factors is in fact optimal. Given that there currently is no evidence to the contrary, the following conjecture has been made (and stated publicly in many talks):

\begin{conjecture}[T\"ornquist]
The isomorphism relation for separably acting factors is universal (from the point of view of Borel reducibility) for orbit equivalence relations induced by a continuous action of the unitary group on a Polish space. In fact, this is already true for isomorphism of $\II_1$ factors.
\end{conjecture}

In other words, the isomorphism relation for $\II_1$ factors attains the upper bound given in Theorem \ref{t.sato3}. A positive solution to this would, in addition to being the ultimate non-classification theorem in the area, give a first example of a ``naturally occurring'' isomorphism relation which realizes the maximal possible complexity that an orbit equivalence relation induced by the unitary group can have.

\section*{Sources}

A great source for learning about von Neumann algebras are the lecture notes of Vaughan Jones \cite{vj}, freely available online at:
\begin{center}
{\tt
http://math.berkeley.edu/~vfr/MATH20909/VonNeumann2009.pdf}
\end{center}

Lecture 1 draws its material from \cite{gkp}, \cite{arv76} and \cite{sakai71}. Lecture 2 is based on \cite{blackadar} and \cite{oanielsen}, and to a lesser extend on \cite{dixmier} and \cite{vj}. In Lecture 3, the discussion of the group measure space von Neumann algebras and crossed products is based on \cite{blackadar} and \cite{vj}, and the discussion of the von Neumann algebra of a non-singular equivalence relation is based on \cite{fm2}. The discussion of hyperfinite von Neumann algebras in Lecture 4 draws on \cite{blackadar} and \cite{connes94}.

\bibliographystyle{alpha}
\bibliography{ASTnotes}

\end{document}